\documentclass[a4paper,reqno,10pt]{amsart}
\usepackage[utf8]{inputenc}
\usepackage{amsfonts}
\usepackage{amsmath}
\usepackage{amsthm}
\usepackage{amssymb}
\usepackage{graphicx} 
\usepackage[dvipsnames]{xcolor}
\usepackage{ytableau}
\usepackage{tikz-cd}
\usepackage{tikz}
\tikzstyle{rectan} = [rectangle, draw, rounded corners, minimum width=2.5cm, minimum height=0.8cm, text centered]
\usetikzlibrary{backgrounds}
\usetikzlibrary{decorations.pathreplacing,}

\usepackage[linktocpage,bookmarksopen=true]{hyperref}

\usepackage{cleveref}

\newtheoremstyle{teoremas}
{12pt}
{13pt}
{\itshape}
{}
{\bfseries}
{}
{.5em}
{}

\theoremstyle{teoremas}
\newtheorem{theorem}{Theorem}[section]
\newtheorem{thm}[theorem]{Theorem}
\newtheorem{corollary}[theorem]{Corollary}
\newtheorem{cor}[theorem]{Corollary}
\newtheorem{lemma}[theorem]{Lemma}
\newtheorem{proposition}[theorem]{Proposition}

\numberwithin{equation}{section}

\newtheoremstyle{definition}
{12pt}
{12pt}
{}
{}
{\bfseries}
{}
{.5em}
{}

\theoremstyle{definition}
\newtheorem{definition}[theorem]{Definition}

\newtheorem{conjecture}[theorem]{Conjecture}

\newtheorem{example}[theorem]{Example}

\newtheorem{remark}[theorem]{Remark}

\newcommand{\M}{\mathsf{M}}

\newcommand{\PP}{\operatorname{PP}}

\usepackage{tikz}
\usetikzlibrary{automata, arrows}
\usepackage[scr=boondoxo]{mathalfa}

\usepackage[shortlabels]{enumitem} 

\hypersetup{
    colorlinks = true,
    linkbordercolor = {white},
    linkcolor = {NavyBlue},
    anchorcolor = {black},
    citecolor = {NavyBlue},
    filecolor = {cyan},
    menucolor = {NavyBlue},
    runcolor = {cyan},
    urlcolor = {NavyBlue}
}

\usepackage[margin=1.30in]{geometry} 
\DeclareMathOperator{\sh}{sh}

\newcommand{\ehr}{\operatorname{ehr}}

\newcommand{\hp}{\mathsf{hp}}

\usepackage{graphicx} 

\usepackage[colorinlistoftodos]{todonotes}

\title[Ehrhart positivity for lattice path matroids]{Ehrhart positivity for lattice path matroids}

\author{Luis Ferroni, Alejandro H. Morales, Greta Panova}

\address{(L. Ferroni)
 University of Pisa, Pisa, Italy
}
\email{luis.ferroni@unipi.it}
\address{(A. H. Morales)
Universit\'e du Qu\'ebec \`a Montr\'eal, Montr\'eal, Canada}
\email{morales\_borrero.alejandro@uqam.ca}

\address{(G. Panova)
 University of Southern California, Los Angeles (CA), United States
}
\email{gpanova@usc.edu}


\thanks{}

\begin{document}
\allowdisplaybreaks

\begin{abstract}
    We prove that all lattice path matroids are Ehrhart positive. This unifies and generalizes numerous results on the Ehrhart positivity of matroids developed over the last two decades. We rely on our previous work on the positivity of order polynomials of fences. Our main result supports the conjecture by Ferroni, Jochemko, and Schr\"oter (2022) on the Ehrhart positivity of positroids. Furthermore, our main result implies that all Schubert matroids are Ehrhart positive, which thus settles a conjecture by Fan and Li (2024), and shows Newton polytopes of certain Schubert and key polynomials are Ehrhart positive. 
\end{abstract}


\maketitle


\section{Introduction}

\subsection{Overview}

A classical theorem of Ehrhart \cite{ehrhart} states that, for every lattice polytope
$P\subseteq \mathbb{R}^n$ of dimension $d$, the counting function
\[
m \longmapsto |mP\cap \mathbb{Z}^n|,
\qquad m\in \mathbb{Z}_{\geq 0},
\]
is given by a polynomial of degree $d$. This polynomial, denoted by
$\ehr_P(t)\in \mathbb{Q}[t]$, is called the \emph{Ehrhart polynomial} of $P$.
It encodes fundamental geometric information about the polytope. For instance,
when $P$ is full-dimensional, the leading coefficient of $\ehr_P(t)$ is the
Euclidean volume of $P$ up to the usual normalization factor, while the
second-highest coefficient records the relative volume of the boundary of $P$.
We refer to the book by Beck and Robins \cite{beck-robins} for a modern treatment of Ehrhart theory.

When the polytope $P$ arises from a combinatorial object, one hopes for more:
ideally, the coefficients of $\ehr_P(t)$ should admit a combinatorial
interpretation. A natural first step towards such an interpretation is to ask
whether these coefficients are nonnegative. If this holds, we say that $P$ is
\emph{Ehrhart positive}. This positivity property is subtle. It fails in
general, and even for highly structured families of lattice polytopes it can be
difficult to establish. We refer to \cite{liu,ferroni-higashitani} for surveys
on Ehrhart positivity. There is also a geometric reason why Ehrhart positivity is a natural property to
seek. The Ehrhart polynomial is a valuation on polytopes, meaning that it
interacts with polyhedral subdivisions through inclusion--exclusion. Moreover,
McMullen \cite{mcmullen} showed that the coefficient of degree $i$ of the
Ehrhart polynomial can be expressed as a weighted sum of the volumes of the
codimension-$i$ faces of $P$, where the weights are determined, non-canonically,
by the corresponding normal cones. From this perspective, positivity of
Ehrhart coefficients can be viewed as a shadow of local geometric positivity
phenomena. 

The present paper concerns Ehrhart positivity for matroid base polytopes. These polytopes form an important and well-studied class of
generalized permutohedra, and their Ehrhart polynomials have attracted
considerable attention. A conjecture of De Loera, Haws, and K\"oppe
\cite{deloera-haws-koppe} predicted that every matroid base polytope is
Ehrhart positive. This conjecture inspired substantial work over nearly two
decades, leading to many partial positive results; see
\cite{postnikov,castillo-liu,castillo-liu2,ferroni_hypersimplex,
ferroni_hooks,jochemko-ravichandran,ferroni-jochemko-schroter,hanely,
fan-li,mcginnis,deligeorgaki-mcginnis-vindas}. Although the conjecture was
ultimately disproved by Ferroni \cite{ferroni}, the persistence of Ehrhart
positivity in many important classes of matroids remains a compelling and
largely unexplained phenomenon.

This question fits naturally into the broader wave of positivity results in
matroid theory that has transformed algebraic combinatorics over the past
decade. Breakthroughs in the so-called Hodge theory of matroids, surveyed for
instance in \cite{kalai,ardila,eur}, have revealed deep geometric structures
behind several positivity phenomena. 

A particularly promising class is that of positroids, a family of matroids associated to the cells of the totally positive Grassmannian introduced by Postnikov; see \cite{Postnikov_positroids}. This class encompasses all
previously known positive instances of the De Loera--Haws--K\"oppe conjecture.
Motivated by this evidence, Ferroni, Jochemko, and Schr\"oter \cite[Conjecture~6.1]{ferroni-jochemko-schroter} formulated the
following conjecture.

\begin{conjecture}[\cite{ferroni-jochemko-schroter}]\label{conj:positroids}
    Positroids are Ehrhart positive.
\end{conjecture}

Although this conjecture is stated for positroids, it is widely expected that a
similar positivity phenomenon should hold for the broader class of
\emph{polypositroids} introduced by Lam and Postnikov; see \cite{lam-postnikov}.

In our previous paper \cite{ferroni-morales-panova}, we studied a family of
matroids called \emph{snake matroids}. These matroids are, up to unimodular
equivalence, positroids, and they satisfy a stronger form of
Ehrhart positivity: their Ehrhart polynomials expand positively in the basis $\{(t-1)^m\}_{m \geq 0}$. On the other hand, it is known that the Ehrhart polynomial of an
arbitrary matroid can be expressed as a signed combination of Ehrhart
polynomials of snake matroids (see \cite[Appendix~A]{ferroni-schroter}). This makes snakes a useful testing ground for
understanding how far one might hope to push positivity from special positroids
towards the entire class.

There are two natural intermediate classes between snake matroids and
positroids: series-parallel matroids and lattice path matroids. These inclusions
are depicted in Figure~\ref{fig:hierarchy}. In fact, snake matroids are exactly
the intersection of the classes of series-parallel matroids and lattice path
matroids.

\begin{figure}[ht]
    \centering
    \scalebox{0.9}{
    \begin{tikzpicture}
        \tikzset{node distance = 2.0cm and 1cm}
        \tikzstyle{arrow} = [-{>[scale=1.8, length=2, width=3.5]}, double]
        \node (snakes)     [rectan]                                    {Snake matroid};
        \node (sp)     [rectan, right=of snakes, yshift= 1.0cm] {Series-parallel matroid};
        \node (LPM)[rectan, right=of snakes, yshift=-1.0cm] {Lattice path matroid};
        \node (pos)        [rectan, right=of sp, yshift=-1.0cm] {Positroid};
        \draw[arrow] (snakes)      |- (sp);
        \draw[arrow] (snakes)      |- (LPM);
        \draw[arrow] (LPM) -| (pos);
        \draw[arrow] (sp)      -| (pos);
    \end{tikzpicture}}
    \caption{From snakes to positroids.}
    \label{fig:hierarchy}
\end{figure}
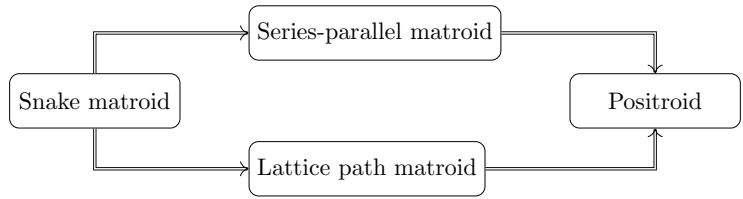

The main result of this paper establishes Ehrhart positivity for one of these
two intermediate classes, while the other remains open (see \cite[Problem~5.7]{ferroni-morales-panova}).

\begin{theorem} \label{thm:main}
    Lattice path matroids are Ehrhart positive.
\end{theorem}

The class of lattice path matroids was first studied by Bonin, De Mier, and Noy \cite{bonin-demier-noy}. Their Ehrhart theory is featured extensively in work of Bidkhori \cite{bidkhori}, Knauer, Mart\'inez-Sandoval, and Ram\'irez-Alfons\'in \cite{knauer-martinez-ramirez},  Benedetti, Knauer,  Valencia \cite{benedetti-knauer-valencia}, and Sanchez \cite{sanchez_MV}.

This result gives new evidence for the conjectural Ehrhart positivity of
positroids. It also has several immediate consequences. Since Schubert matroids
are (up to an isomorphism) a special class of lattice path matroids, our theorem implies that all Schubert matroids are
Ehrhart positive, thereby settling a conjecture of Fan and Li
\cite[Conjecture~1.6]{fan-li}. It also generalizes the Ehrhart positivity of
panhandle matroids recently established by Deligeorgaki, McGinnis, and
Vindas-Mel\'endez \cite[Theorem~1.3]{deligeorgaki-mcginnis-vindas}.
Furthermore, our result recovers the Ehrhart positivity of hypersimplices and
minimal matroids treated by Ferroni in \cite{ferroni} and \cite{ferroni_hooks},
respectively, the class of matroids covered by McGinnis in \cite{mcginnis}, the case of Catalan matroids addressed by Chen, Li, and Yao
\cite{chen-li-yao}, and the snake matroids studied in
\cite{ferroni-morales-panova}.

Another relevant problem in Ehrhart positivity was posed by Monical, Tokcan, and Yong \cite[Conjecture~5.19]{monical-tokcan-yong}. These authors conjectured that a class of polytopes called \emph{Schubitopes}, which are Newton polytopes of Schubert polynomials, are Ehrhart positive.  While Schubitopes are a subclass of generalized permutohedra, they are rarely matroids.  They are Minkowski sums of Schubert matroids \cite[Thm. 7]{fink-meszaros-stdizier}.  Theorem~\ref{thm:main} shows the conjecture is true when the Schubitope is one Schubert matroid. However, recently, Li and St. Dizier \cite{LiStDizierCounterSchubi} found a counterexample to this conjecture that is the Minkowski sum of two hypersimplices. 

\subsection*{Acknowledgments}

The authors thank Oberwolfach and the organizers of the Enumerative Combinatorics workshop in January 2026, where part of this work was carried out. The authors also acknowledge many fruitful conversations about fence posets and Ehrhart theory with the participants of the American Institute of Mathematics workshop on Ehrhart polynomials held in May 2022. We also thank Dania Morales, whose thesis was an inspiring source of ideas that ultimately led to this project to succeed, and for her encouragement during the 2026 Joint Math Meetings, and David Speyer for very useful conversations about series-parallel subdivisions in matroid theory, Jonathan Boretsky, Yakob Kahane, and Sophie Rehberg for helpful comments and suggestions. AHM is partially supported by an  NSERC Discovery grant RGPIN-2024-06246. GP is partially supported by an NSF grant CCF-2302174.

\section{Background}

We will assume that the reader is familiar with the basic terminology on matroid polyhedra, subdivisions, and valuations along the lines of \cite[Sections~2.3 and 2.4]{ferroni-schroter}

\subsection{Schubert matroids and lattice path matroids}

We now recall the basic definitions and properties of lattice path matroids. Fix two non-negative integers $k \leq n$. We consider lattice paths in the plane $\mathbb{R}^2$ starting at $(0,0)$, ending at $(n-k,k)$, and consisting of exactly $n$ steps, each of which is either $+(1,0)$ or $+(0,1)$. For instance, Figure~\ref{fig:lattice-path} depicts two such lattice paths for the parameters $k=5$ and $n=10$.

\begin{figure}[ht]
    \centering

        
        \includegraphics[scale=0.7]{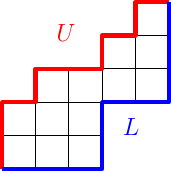}
        
    \caption{Two lattice paths for $k=5$ and $n=10$. The blue path stays below the red path.}
    \label{fig:lattice-path}
\end{figure}

It is convenient to encode such a lattice path as a word of length $n$ in the alphabet $\{\text{N},\text{E}\}$, where the $i$-th letter is N if the $i$-th step is $+(0,1)$, and E if the $i$-th step is $+(1,0)$. For example, the red path in Figure~\ref{fig:lattice-path} is represented by the word
\(
U=\text{NNENEENENE}.
\)
Equivalently, one may record the subset $s(U)\subseteq [n]$ consisting of the positions at which the letter N occurs. Thus, for the red path $U$ in Figure~\ref{fig:lattice-path}, we have
\[
s(U)=\{1,2,4,7,9\}.
\]

Now fix $k$ and $n$, and let $L$ and $U$ be two lattice paths with these parameters. We say that $L$ \emph{lies below} $U$ if, for every $m=1,\ldots,n$, one has
\begin{equation}\label{eq:ineq-lattice-paths}
    \left|s(L) \cap [m]\right| \leq \left| s(U) \cap [m]\right|\,.
\end{equation}
This condition has the expected geometric meaning in the rectangular grid: the path $L$ never rises above the path $U$. This is illustrated in Figure~\ref{fig:lattice-path}, where the blue path lies below the red path. In that example,
\[
s(L)=\{4,5,8,9,10\}
\qquad\text{and}\qquad
s(U)=\{1,2,4,7,9\}.
\]

Whenever $L$ lies below $U$, we also say that $U$ lies above $L$. If, in addition, the inequality in \eqref{eq:ineq-lattice-paths} is strict for every $m=1,\ldots,n-1$, then we say that $L$ lies \emph{strictly below} $U$, or equivalently that $U$ lies \emph{strictly above} $L$.

\begin{theorem}[\cite{bonin-demier}]\label{thm:lattice_path_matroid}
    Fix $0\leq k\leq n$, and let $L$ and $U$ be two lattice paths such that $L$ lies below $U$. Then
    \[
        \mathscr{B}[L,U]
        :=
        \{s(P): P\text{ is a lattice path lying above $L$ and below $U$}\}
    \]
    is the set of bases of a matroid on $[n]$ of rank $k$.
\end{theorem}

The matroid described in Theorem~\ref{thm:lattice_path_matroid} is denoted by $\M[L,U]$. We call it the lattice path matroid with \emph{lower path} $L$ and \emph{upper path} $U$. The foundations of the theory of lattice path matroids are developed in \cite{bonin-demier} and \cite{bonin-demier-structural}.

Schubert matroids arise from generic points in Schubert cells of the flag variety. They form, up to isomorphism, a subclass of lattice path matroids. The following statement gives the precise characterization of Schubert matroids within this framework. We use the notation $\M[U]:=\M[L,U]$ for lattice path matroids of rank $k$ on $n$ elements whose lower path is
\(
L=\text{\normalfont E\dots EN\dots N},
\)
namely the lower-right boundary of the rectangular grid with vertices $(0,0)$, $(0,k)$, $(n-k,0)$, and $(n-k,k)$.

\begin{theorem}\label{teo:schubert-is-lpm}
    A matroid $\M$ is a Schubert matroid if and only if it is isomorphic to a lattice path matroid of the form $\M[U]$.
\end{theorem}

\begin{proof}
    This is precisely \cite[Corollary~4.4]{bonin-demier-structural}.
\end{proof}

\subsection{Subdivisions of lattice path matroids into snakes}

In \cite{luis_delannoy} Ferroni described how Schubert matroids can be subdivided into snake matroids. The cells in the subdivision described by Ferroni are described via a notion of \emph{admissible Delannoy path}. We introduce a slight generalization that covers all lattice path matroids rather than only the Schubert matroids. 

\begin{definition}\label{def:adm_path}
    Let $\M = \M[L,U]$ be a lattice path matroid on $[n]$ of rank $r$. The \emph{admissible Delannoy paths} associated to $\M$ are all the paths starting at $(1,1)$ and ending at $(n-r,r)$, having steps of the form $+(1,0)$, $+(0,1)$ and $+(1,1)$, and satisfying the following requirements:
    \begin{enumerate}[(i)]
        \item The paths stay within the lattice path representation of $\M$, i.e., they do not go above $U$ nor below $L$.
        \item An intermediate step of the form $+(0,1)$ is valid only if it does not yield a vertical overlap with the upper path $U$.
        \item An intermediate step of the form $+(1,0)$ is valid only if it does not yield a horizontal overlap with the lower path $L$.
        \item An intermediate step of the form $+(1,1)$ is valid only when both the steps $+(0,1)$ and $+(1,0)$ are valid.
    \end{enumerate}
\end{definition}

\begin{remark}\label{rem:delannoy_shift}
    The above definition is equivalent to the following interpretation. Shift the Delannoy path by $(-1/2,-1/2)$, so that the grid steps are the centers of the squares of shape between $L$ and $U$, with initial point at $(1/2,1/2)$ and ending point $(n-r-1/2,r-1/2)$. The conditions above are exactly equivalent to having the shifted path be entirely contained inside the shape and not touching the boundaries. 
\end{remark}

Observe that when $\M$ is a Schubert matroid, i.e., when $L$ is the lower-right border of the grid, the conditions on the overlap are superfluous for $L$, and therefore the above definition matches \cite[Definition~3.1]{luis_delannoy}.

\begin{example}
    Consider the lattice path matroid given by $U=\{1,2,4,7,9\}$ and $L=\{4,5,8,9,10\}$. From left to right, we have in Figure~\ref{fig:delannoy-non-examples} three Delannoy lattice paths that are \emph{not} admissible, followed by two admissible examples. The first path has a vertical overlap with $U$ at the fourth step, and thus violates condition (ii) of the definition. The second and the third path make a diagonal step in a forbidden position: in the second diagram, the diagonal occurs when doing a horizontal step overlaps $L$ horizontally, whereas in the third diagram there is a diagonal step where a vertical step would yield a vertical overlap with $U$. The two remaining examples are admissible: notice that in the fourth of the five paths we do have an overlap with $U$ but it is not vertical. 

    \begin{center}
    \begin{figure}[ht]
    \includegraphics[scale=0.7]{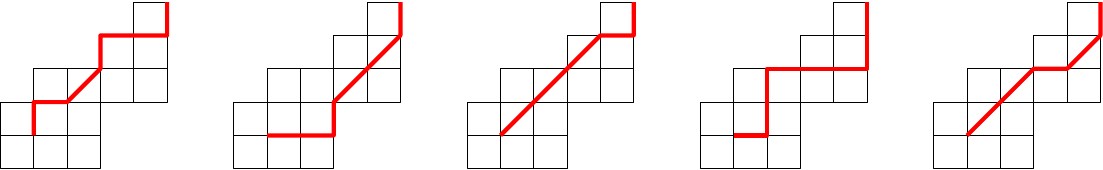}

    \caption{From left to right, three non-examples followed by two examples.}\label{fig:delannoy-non-examples}
    \end{figure}
\end{center}
\end{example}

\subsection{Skew shapes, lattice path matroids, and Ehrhart polynomials}

A convenient way of encoding a lattice path matroid is via a \emph{skew shape} denoted by $\lambda/\mu$ for partitions $\lambda$ and $\mu$. For example, the lattice path matroid in Figure~\ref{fig:lattice-path} corresponds to the skew shape $\lambda/\mu$ for $\lambda = (5,5,5,3,3)$ and $\mu = (4,3,1)$. 
A \emph{ribbon} or \emph{border strip} is a (not necessarily connected) skew shape whose Young diagram does not contain a $2\times 2$ square. A lattice path matroid associated to a \emph{connected} ribbon shape is called a \emph{snake matroid}. A \emph{kitty corner} of a skew shape is a corner that is the intersection of two cells in the shape.

The following result follows directly from \cite[Theorem~3.3]{luis_delannoy}. To avoid unnecessary technicalities, we restrict ourselves only to connected matroids. 

\begin{theorem}\label{thm:bijection-delannoy-snakes}\cite[Theorem~3.3]{luis_delannoy}
    Let $\M$ be a connected lattice path matroid of rank $r$ on $[n]$. There exists a subdivision $\mathcal{S}$ of $\mathscr{P}(\M)$ into base polytopes of snake matroids (of the same size and rank as $\M$). Moreover, there is a bijection $\varphi$ between the internal faces of $\mathcal{S}$ and the set of all admissible Delannoy paths associated to $\M$. This bijection has the following properties:
    \begin{enumerate}[\normalfont(i)]
        \item The facets of the subdivision $\mathcal{S}$ correspond to paths with no diagonal steps.
        \item More generally, for each matroid $\mathsf{N}$ corresponding to an internal face of $\mathcal{S}$, the number of connected components $c(\mathsf{N})$ of $\mathsf{N}$ satisfies that $c(\mathsf{N})-1$ is the number of diagonal steps of the Delannoy path associated to $\mathsf{N}$.
    \end{enumerate}
\end{theorem}

The subdivision of a lattice path matroid into snakes was studied in detail by Bidkhori in her PhD Thesis \cite{bidkhori}. For a streamlined presentation, we suggest \cite[Appendix~A]{ferroni-schroter} and, in the context of Ehrhart polynomials, Dania Morales's PhD Thesis \cite[Theorem~4.2.5]{dania}. What results from the above theorem is that the admissible Delannoy paths for a lattice path matroid $\M$ associated with a skew shape $\lambda/\mu$ correspond exactly to the snakes that appear in the aforementioned subdivision of $\M$ into snakes. Moreover, from the bijection it is straightforward to see that if we translate the admissible Delannoy paths by the vector $(-1/2, -1/2)$ as in Remark~\ref{rem:delannoy_shift} (so that they start at $(1/2, 1/2)$ and end at $(n-r-1/2, r-1/2)$ rather than start at $(1,1)$ and end at $(n-r,r)$) we obtain the paths in $D(\lambda/\mu)$. For a skew shape $\lambda/\mu$, let us denote by $\mathcal{D}(\lambda/\mu)$ the set of all admissible Delannoy paths corresponding to the lattice path matroid associated to $\lambda/\mu$, but translated by $(-1/2,-1/2)$.  See Figure~\ref{fig:snakes-and-delannoy} for examples of these paths. 

\begin{figure}[ht]
\centering
\includegraphics[scale=1.3]{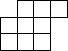}
\,\,
\includegraphics[scale=1.3]{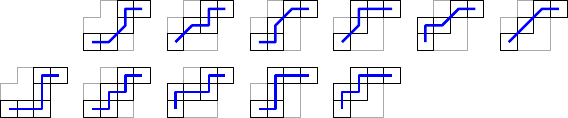}
\caption{Left: the skew shape $\lambda/\mu=433/1$. Right: Ribbon shapes appearing in Theorem~\ref{thm:bijection-delannoy-snakes} and the associated Delannoy paths in $\mathcal{D}(\lambda/\mu)$. }\label{fig:snakes-and-delannoy}
\end{figure}

\begin{theorem}\label{thm:lpm-from-snakes}
Let $\M$ be a lattice path matroid associated to the connected skew shape $\lambda/\mu$. The Ehrhart polynomial of $\mathscr{P}(\M)$ is given by:
\begin{equation}
\ehr(\mathscr{P}(\M),t) \,=\, \sum_{\pi} (-1)^{c(\pi)-1} \ehr(\mathscr{P}(\M(\pi)),t),
\end{equation}
where the sum is over ribbon shapes within the Young diagram that contain the lower-leftmost and upper-rightmost squares of $\lambda/\mu$, where consecutive components intersect in kitty corners in the interior of $\lambda/\mu$, and $c(\pi)$ is the number of connected components of $\pi$.
\end{theorem}

The above theorem is implicit in the proof of \cite[Lemma~4.3.6]{bidkhori}, and can also be found in \cite[Theorem~4.2.5]{dania}. In light of the bijection between admissible Delannoy paths and the snakes subdividing a lattice path matroid, we will henceforth use the notation $\M(\gamma)$ to denote the snake that corresponds to the Delannoy path $\gamma\in \mathcal{D}(\lambda/\mu)$.

The following is a straightforward reformulation of the last theorem after employing the bijection appearing in Theorem~\ref{thm:bijection-delannoy-snakes}.

\begin{cor}\label{cor:lpm-from-delannoy}
Let $\M$ be a lattice path matroid associated to the connected skew shape $\lambda/\mu$. The Ehrhart polynomial of $\mathscr{P}(\M)$ is given by:
\label{cor: 1st reformulation Ehrh LPMs}
\begin{equation}
\ehr(\mathscr{P}(\M),t) \,=\,  \sum_{\gamma \in \mathcal{D}(\lambda/\mu)} (-1)^{d(\gamma)} \ehr(\mathscr{P}(\M(\gamma)),t),
\end{equation}
where $d(\gamma)$ denotes the number of diagonals of the Delannoy path.
\end{cor}

\begin{remark}
    In fact, Theorem~\ref{thm:lpm-from-snakes} and Corollary~\ref{cor:lpm-from-delannoy} hold at the level of \emph{valuative invariants} of matroids (see \cite{ardila-fink-rincon,derksen-fink,ferroni-schroter}). For example, since the Tutte polynomial of a matroid is a valuative invariant too, one can write the Tutte polynomial of a lattice path matroid exactly as the same combination of Tutte polynomials of the snakes appearing in the subdivision appearing in Theorem~\ref{thm:bijection-delannoy-snakes}. 
\end{remark}

\begin{example} \label{ex: master example excited version}
For the lattice path matroid  $\M$ associated to $\lambda/\mu=(433)/(1)$ we have that 
\[
\ehr(\mathscr{P}(\M),t)= \frac{47}{180}t^6 + \frac{109}{60}t^5 + \frac{383}{72}t^4 + \frac{17}{2}t^3 + \frac{2851}{360}t^2 + \frac{251}{60}t + 1
\]
\[
\ytableausetup{smalltableaux}
\begin{array}{l|crr}  \hline
\text{path } \gamma & d(\gamma) &  (-1)^{d(\gamma)} &  \ehr(\mathscr{P}(\M(\gamma))),t) \\ \hline 
\ydiagram{2+2,2+1,3} & 0 & 1 &  \frac{13}{360} \, t^{6} + \frac{49}{120} \, t^{5} + \frac{67}{36} \, t^{4} + \frac{35}{8} \, t^{3} + \frac{2017}{360} \, t^{2} + \frac{223}{60} \, t + 1\\ \hline
\ydiagram{2+2,1+2,2} & 0 &  1 &  \frac{61}{720} \, t^{6} + \frac{61}{80} \, t^{5} + \frac{413}{144} \, t^{4} + \frac{277}{48} \, t^{3} + \frac{2357}{360} \, t^{2} + \frac{119}{30} \, t + 1\\
\ydiagram{2+2,2+1,2} & 1 & -1 &  \frac{1}{6} \, t^{5} + \frac{5}{4} \, t^{4} + \frac{11}{3} \, t^{3} + \frac{21}{4} \, t^{2} + \frac{11}{3} \, t + 1\\ \hline 
\ydiagram{2+2,3,1} & 0 & 1 & \frac{1}{18} \, t^{6} + \frac{17}{30} \, t^{5} + \frac{169}{72} \, t^{4} + \frac{61}{12} \, t^{3} + \frac{439}{72} \, t^{2} + \frac{77}{20} \, t + 1\\
\ydiagram{2+2,1+2,1} & 1 & -1 & \frac{5}{24} \, t^{5} + \frac{35}{24} \, t^{4} + \frac{97}{24} \, t^{3} + \frac{133}{24} \, t^{2} + \frac{15}{4} \, t + 1
 \\ \hline 
\ydiagram{1+3,1+1,2} & 0 & 1 & \frac{13}{360} \, t^{6} + \frac{49}{120} \, t^{5} + \frac{67}{36} \, t^{4} + \frac{35}{8} \, t^{3} + \frac{2017}{360} \, t^{2} + \frac{223}{60} \, t + 1 \\
\ydiagram{2+2,1+1,2} & 1 & -1 & \frac{1}{6} \, t^{5} + \frac{5}{4} \, t^{4} + \frac{11}{3} \, t^{3} + \frac{21}{4} \, t^{2} + \frac{11}{3} \, t + 1
 \\ \hline 
\ydiagram{1+3,2,1} &0 & 1 &  \frac{7}{144} \, t^{6} + \frac{121}{240} \, t^{5} + \frac{307}{144} \, t^{4} + \frac{227}{48} \, t^{3} + \frac{419}{72} \, t^{2} + \frac{113}{30} \, t + 1\\
\ydiagram{1+3,1+1,1} & 1 & -1 & \frac{1}{8} \, t^{5} + \frac{25}{24} \, t^{4} + \frac{79}{24} \, t^{3} + \frac{119}{24} \, t^{2} + \frac{43}{12} \, t + 1\\
\ydiagram{2+2,2,1} & 1 & -1 & \frac{1}{6} \, t^{5} + \frac{5}{4} \, t^{4} + \frac{11}{3} \, t^{3} + \frac{21}{4} \, t^{2} + \frac{11}{3} \, t + 1 \\
\ydiagram{2+2,1+1,1} & 2 & 1 & \frac{1}{2} \, t^{4} + \frac{5}{2} \, t^{3} + \frac{9}{2} \, t^{2} + \frac{7}{2} \, t + 1 \\
\hline 
\text{total} & & &   \frac{47}{180}t^6 + \frac{109}{60}t^5 + \frac{383}{72}t^4 + \frac{17}{2}t^3 + \frac{2851}{360}t^2 + \frac{251}{60}t + 1 \\[1mm] \hline 
\end{array}
\]
\end{example}

\section{The Ehrhart polynomial of a lattice path matroid}

In this section we will prove the main theorem. Since disconnected lattice path matroids are direct sums of connected lattice path matroids, and since the Ehrhart polynomial of matroids is multiplicative under direct sums, we will assume throughout that $\lambda/\mu$ is a connected skew shape with non-empty rows whose associated lattice path matroid is $\M = \M(\lambda/\mu)$, which is connected.  Recall that
    \[ \mathcal{D}(\lambda/\mu) := \{\text{admissible Delannoy paths for $\lambda/\mu$ translated by $(-1/2,-1/2)$}\}.\]
We further introduce the following useful notation:
    \[ \mathcal{L}(\lambda/\mu) := \{\gamma \in \mathcal{D}(\lambda/\mu) \text{ with no diagonal steps}\}.\]

\subsection{Delannoy paths, high peaks, and grouping terms}

In this section we will propose a way of grouping admissible Delannoy paths for a lattice path matroid $\M = \M(\lambda/\mu)$. What we are about to do here is very reminiscent to the proof of \cite[Theorem~4.3]{luis_delannoy}, which led to a strong positivity property of Speyer's $g$-polynomial on Schubert matroids. This grouping relies on a bijection between Delannoy paths and certain marked lattice paths of the shape $\lambda/\mu$. This bijection for the case of $\mu=\varnothing$  appeared in the Schubert calculus literature in work of  Naruse--Okada \cite[Proposition~3.13]{NO} and Morales--Pak--Panova  \cite[Lemma 6.15]{MPP1},  \cite[Proposition~5.1]{MPP4} . See Remark~\ref{remark: connection Excited diagrams} for more details on this connection.

\begin{definition}
Let $\lambda/\mu$ be a connected skew shape. Let
$\ell$ be the number of rows, and write cells as $(i,j)$, with
$i$ increasing from top to bottom and $j$ increasing from left to right.
Define a partition $\nu$ by
\[
\nu_i =
\begin{cases}
\lambda_{i+1}-1, & 1\leq i<\ell,\\
\mu_\ell, & i=\ell.
\end{cases}
\]
The \emph{minimum lattice path} associated to $\lambda/\mu$, denoted
$\gamma_{\min}\in \mathcal{L}(\lambda/\mu)$, is the north-east path whose associated ribbon is
$\lambda/\nu$.

Equivalently, $\gamma_{\min}$ is obtained by starting at the lower-leftmost
cell of $\lambda/\mu$ and moving east whenever possible, moving north only
when an east step would leave the skew shape. Thus, in row $i$, the cells of
$\gamma_{\min}$ are precisely
\[
\{(i,j): \nu_i<j\leq \lambda_i\}.
\]
\end{definition}

\begin{figure}[ht]
\centering









\includegraphics[scale=1.4]{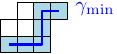}

\caption{The minimum path $\gamma_{\min}$ for $\lambda/\mu=433/1$.}
\label{fig:gamma-min}
\end{figure}

\begin{definition}
    Given a path $\gamma\in\mathcal{L}(\lambda/\mu)$, a {\em peak} is pair of an up step immediately followed by a right step in $\gamma$. The {\em high peaks} of $\gamma$ are the peaks of $\gamma$ that are not also peaks of $\gamma_{\min}$. We denote by $\mathsf{hp}(\gamma)$, the set of high peaks of $\gamma$. 
\end{definition}

\begin{figure}[ht]
\includegraphics[scale=1.4]{skew_shape_running_example.pdf}
\,\,    \includegraphics[scale=1.4]{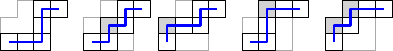}
    \caption{The five paths in $\mathcal{L}(\lambda/\mu)$ (gray) with high peaks distinguished and the associated connected ribbons for the skew shape $\lambda/\mu=433/1$.} \label{fig: excited paths}
    \label{fig: delannoy paths}
\end{figure}

\begin{definition}
A {\em marked path} is a pair $(\gamma,S)$ where $\gamma \in \mathcal{L}(\lambda/\mu)$ and $S\subset \mathsf{hp}(\gamma)$. Given a marked path $(\gamma,S)$, we denote by $\varphi(\gamma,S)$ the Delannoy path in $\mathcal{D}(\lambda/\mu)$ obtained by replacing the high peaks in $S$ by a diagonal step $(1,1)$.
\end{definition}

The next result  shows that each Delannoy path in $\mathcal{D}(\lambda/\mu)$ can be obtained uniquely from a marked path by replacing the marked high peaks into diagonal steps. The case of $\mu=\varnothing$ appeared in the Schubert literature which extends an idea of Naruse and Okada \cite[Proposition~3.13]{NO} (see also \cite[Proposition~5.1]{MPP4}).

\begin{lemma} \label{thm: from delannoy to marked paths}
There is a natural bijection between admissible Delannoy paths in $\lambda/\mu$ and marked north-east lattice paths $(\gamma,S)$ such that $S\subseteq \hp(\gamma)$: 
\begin{equation} \label{eq: decom skew shape}
\mathcal{D}(\lambda/\mu) \xleftarrow{\varphi } \{ (\gamma,S) \mid \gamma \in \mathcal{L}(\lambda/\mu), S \subseteq \mathsf{hp}(\gamma)\},
\end{equation}
In particular $\#\mathcal{D}(\lambda/\mu) =\sum_{\gamma \in \mathcal{L}(\lambda/\mu)} 2^{\#\mathsf{hp}(\gamma)}$.
\end{lemma}


\begin{proof}
By Definition~\ref{def:adm_path}, Remark~\ref{rem:delannoy_shift} and Figure~\ref{fig:snakes-and-delannoy} we see that the admissible [shifted] Delannoy paths are the ones that start at the bottommost leftmost square of $\lambda/\mu$, end at the topmost rightmost square, and are entirely contained inside the shape. In particular, the diagonal steps of the paths cannot touch the corners of the upper or lower boundary of $\lambda/\mu$. Touching a corner of the lower boundary (i.e. $\lambda$) is equivalent to the diagonal step cutting a peak in $\gamma_{\min}$. Touching the corner of the upper boundary implies touching a corner of the shape $\mu$. Since this is not allowed, a diagonal step of the Delannoy path is contained in a $2 \times 2$ box, going through its center, which sits entirely in $\lambda/\mu$.  Let the Delannoy path be $\delta$ and the endpoints of its diagonal steps  be $A_1=(a_1,b_1)$ and $C_1=(a_1+1,b_1+1)$, $A_2=(a_2,c_2)$ and $C_2=(a_2+1,b_2+1)$ etc, then $B_i=(a_i,b_i+1)$ and $D_i=(a_i+1,b_i)$ are inside $\lambda/\mu$ for each $i$. We can write the path $\delta = \gamma_1 A_1C_1 \gamma_2 A_2C_2 \gamma_3 \cdots$, where the paths $\gamma_i$ (which can be empty) do not have diagonal steps. We then make $\gamma = \gamma_1 A_1D_1C_1 \gamma_2 A_2D_2C_2 \cdots$ and $S = \{D_1,D_2,\ldots\}$, where we replace each diagonal step $A_iC_i$ by the peak of North and East steps $A_iD_iC_i$ and mark that peak in the set $S$. By the above consideration the path $\gamma\in \mathcal{L}(\lambda/\mu)$ and each $D_i$ is a high peak since its diagonal square $B_i$ is inside the shape. By construction we have $\varphi(\gamma,S) = \delta$. This completes the proof.  
\end{proof}

\begin{figure}
    \includegraphics[width=2in]{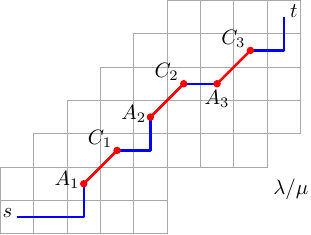}
    \caption{Decomposition of Delannoy paths in the proof of Lemma~\ref{thm: from delannoy to marked paths}.}
\end{figure}

\begin{remark} \label{remark: connection Excited diagrams}
The decomposition in \eqref{eq: decom skew shape} for the case when $\mu=\varnothing$ appeared in the context of Schubert calculus. Namely, Naruse announced in 2014 a formula for the number of standard Young tableaux of skew shape $\lambda/\mu$ as a nonnegative sum of products of hook-lengths. The sum is indexed by certain subsets of the cells of the {\em Young diagram} of $\lambda$ called {\em excited diagrams} of the skew shape introduced by Naruse and Ikeda \cite{IN1} that appeared in their study of structure constants of {\em equivariant Schubert calculus}. There is a generalization of this story in the context of {\em equivariant K-theory} that uses a larger class of subsets of the Young diagram of $\lambda$ called {\em generalized excited diagrams} introduced by Graham--Kreiman \cite{GK} and Ikeda--Naruse \cite{IN2}.  When the skew shape is a connected ribbon with underlying shape $\lambda/\nu$, the excited diagrams  and generalized excited diagrams of this shape are exactly the supports of the lattice paths $\mathcal{L}(\lambda)$ and Delannoy paths $\mathcal{D}(\lambda)$, respectively. Naruse and Okada showed in \cite[Prop. 3.13]{NO} a connection between generalized excited diagrams that for ribbon shapes is exactly the case when $\mu=\varnothing$ of \eqref{eq: decom skew shape}
\begin{equation} \label{eq:decomp straight shape case}
\mathcal{D}(\lambda) \,=\, \{ \varphi(\gamma,S) \mid \gamma \in \mathcal{L}(\lambda), S \subseteq \mathsf{hp}(\gamma)\}.
\end{equation}
In fact, it was this special case the lead us to look for the skew version of this statement in Lemma~\ref{thm: from delannoy to marked paths}.
It would be interesting to see if this skew version is of interest in Schubert calculus. See for example the marked bumpless pipe dreams in \cite{AWBPD}.
\end{remark}

As advertised above, the above result suggests a way of grouping our admissible Delannoy paths. Concretely, two admissible Delannoy paths in $\mathcal{D}(\lambda/\mu)$ are in the same group if the lattice path $\gamma\in \mathcal{L}(\gamma/\mu)$ is the same in their image $(\gamma,S)$ under the bijection appearing in the preceding lemma.

\begin{definition}
    Given a path $\gamma$ in $\mathcal{L}(\lambda/\mu)$, let 
    \[
    \ehr^{\pm}(\mathscr{P}(\M(\gamma)),t) := \sum_{S \subseteq \mathsf{hp}(\gamma)} (-1)^{\#S} \ehr(\mathscr{P}(\M(\varphi(\gamma,S)),t).
    \]
\end{definition}

To avoid confusions, we emphasize that the polynomial $\ehr^{\pm}(\mathscr{P}(\M(\gamma)),t)$ is not an Ehrhart polynomial, but rather a signed sum of Ehrhart polynomials.

\begin{theorem} \label{thm: 2nd reformulation Ehrh LPMs}
The Ehrhart polynomial of the lattice path matroid associated to the skew shape $\lambda/\mu$ can be computed as follows:
\begin{equation}
\ehr(\mathscr{P}(\M(\lambda/\mu)),t) \,=\, \sum_{\gamma \in \mathcal{L}(\lambda/\mu)} \ehr^{\pm}(\mathscr{P}(\M(\gamma)),t).
\end{equation}
\end{theorem}

\begin{proof}
The result follows by combining Corollary~\ref{cor: 1st reformulation Ehrh LPMs} with Lemma~\ref{thm: from delannoy to marked paths}. 
\end{proof}

\begin{remark}
    The above result holds at the level of valuative invariants. In other words, if $f$ is a valuative invariant of matroids, by defining:
    \[
    f^{\pm}(\M(\gamma)) := \sum_{S \subseteq \mathsf{hp}(\gamma)} (-1)^{\#S} f(\M(\varphi(\gamma,S))),
    \]
    one can compute:
    \[f(\M(\lambda/\mu)) \,=\, \sum_{\gamma \in \mathcal{L}(\lambda/\mu)} f^{\pm}(\M(\gamma)).\]
\end{remark}

\begin{example} \label{ex: master example excited version2}
Let us illustrate Theorem~\ref{thm: 2nd reformulation Ehrh LPMs}. For the lattice path matroid  $\M$ associated to $\lambda/\mu=433/1$ we have that 
\[
\ehr(\mathscr{P}(\M),t)= \frac{47}{180}t^6 + \frac{109}{60}t^5 + \frac{383}{72}t^4 + \frac{17}{2}t^3 + \frac{2851}{360}t^2 + \frac{251}{60}t + 1
\]
\[
\ytableausetup{smalltableaux}
\begin{array}{r|l}  \hline
\text{path } \gamma & \ehr^{\pm}(\mathscr{P}(\M(\gamma)),t) \\ \hline 
\gamma_{\min} \,\,\ydiagram{2+2,2+1,3} &  \frac{13}{360} \, t^{6} + \frac{49}{120} \, t^{5} + \frac{67}{36} \, t^{4} + \frac{35}{8} \, t^{3} + \frac{2017}{360} \, t^{2} + \frac{223}{60} \, t + 1\\[3pt]
\gamma_2 \,\,\begin{ytableau}
\none &  \none & &  \\
\none & *(gray!50) &   \\
 &   \\
\end{ytableau}
& \frac{61}{720} \, t^{6} + \frac{143}{240} \, t^{5} + \frac{233}{144} \, t^{4} + \frac{101}{48} \, t^{3} + \frac{467}{360} \, t^{2} + \frac{3}{10} \, t\\[3pt]
\gamma_3 \,\,\begin{ytableau}
\none &  \none & &  \\
 *(gray!50) & &    \\
    \\
\end{ytableau}
& \frac{1}{18} \, t^{6} + \frac{43}{120} \, t^{5} + \frac{8}{9} \, t^{4} + \frac{25}{24} \, t^{3} + \frac{5}{9} \, t^{2} + \frac{1}{10} \, t\\[3pt]
\gamma_4 \,\,\begin{ytableau}
\none &  *(gray!50) &  &  \\
\none &    \\
 &   \\
\end{ytableau}
& \frac{13}{360} \, t^{6} + \frac{29}{120} \, t^{5} + \frac{11}{18} \, t^{4} + \frac{17}{24} \, t^{3} + \frac{127}{360} \, t^{2} + \frac{1}{20} \, t  \\[3pt]

\gamma_5 \,\,\begin{ytableau}
\none &  *(gray!50) &  &  \\
*(gray!50) &    \\
   \\
\end{ytableau}
& \frac{7}{144} \, t^{6} + \frac{17}{80} \, t^{5} + \frac{49}{144} \, t^{4} + \frac{13}{48} \, t^{3} + \frac{1}{9} \, t^{2} + \frac{1}{60} \, t \\ \hline 
\text{total} &  \frac{47}{180}t^6 + \frac{109}{60}t^5 + \frac{383}{72}t^4 + \frac{17}{2}t^3 + \frac{2851}{360}t^2 + \frac{251}{60}t + 1 \\ \hline 
\end{array}
\]
\end{example}

\subsection{From Ehrhart polynomials to plane partitions}

For a skew shape $\lambda/\mu$, a \emph{plane partition} is a filling of the Young diagram of $\lambda/\mu$ with positive integers that is weakly decreasing in rows and columns. For a skew shape $\theta=\lambda/\mu$, let $\PP_{\theta}(t)$ be the number of plane partitions of shape $\theta$ with entries $0, 1,\ldots,t$. Let $\Omega(\theta;t)$ be the order polynomial of the poset corresponding to $\theta$, which is the same as counting plane partitions with entries $1,\ldots,t$ when $\theta$ is a skew shape. By definition we have that 
$$\PP_\theta(t) = \Omega(\theta;t+1),$$
and we will use them interchangeably throughout.  The following result, proved in our prequel \cite{ferroni-morales-panova}, is one of the crucial positivity ingredients needed in the main proof.

\begin{thm}[{\cite[Corollary~5.1]{ferroni-morales-panova}}]
\label{coro:fences-order-positive-main}
    Let $\M$ be a snake matroid, associated to the ribbon shape $\theta$. Then, the following equality holds:
    \[\ehr(\mathscr{P}(\M),t-1)=\Omega(\theta;t).\] Moreover, these polynomials have nonnegative coefficients.
\end{thm}

The plane partition interpretation will be instrumental to deduce the main theorem of the present article. Until now, we have stated results that work \emph{mutatis mutandis} for any valuation. However, by relying on this fact we can prove the following crucial result. Recall that an {\em order filter} or dual order ideal of a poset $P$ is an upward closed subset of $P$.

\begin{theorem}\label{thm: ehr +- fence is nonnegative }
For a connected skew shape $\lambda/\mu$ and a path $\gamma$ in $\mathcal{D}(\lambda/\mu)$ other than $\gamma_{\min}$, we have that
$$
\ehr^{\pm}(\mathscr{P}(\M(\gamma)),t) \,=\, \sum_{
F\subset \gamma \setminus \langle \hp(\gamma)\rangle} \Omega(\gamma \setminus F ;t),
$$
where we view $\gamma$ as a poset and the sum is over order filters $F$ of it. In particular $\ehr^{\pm}(\mathscr{P}(\M(\gamma)),t)$, has nonnegative coefficients.
\end{theorem}

\begin{proof}
For a path $\gamma$ in $\mathcal{L}(\lambda/\mu)$ viewed as a ribbon shape, $S \subset \hp(\gamma)$, and for $t\in \mathbb{N}$, the quantity $\ehr(\mathscr{P}(\M(\varphi(\gamma,S)),t)$ counts the number of plane partitions of shape $\gamma \setminus S$  with entries $0,\ldots,t$. This is equivalent to plane partitions of $\gamma$ with the entries in the cells of $S$ having values $t$ and can be written as $\PP_\gamma(t;S=t)$. Thus, grouping the plane partitions by which high peaks have values equal to $t$ (set $T$), we uncover the inclusion-exclusion identity
\begin{align*}
    \ehr^{\pm}(\mathscr{P}(\M(\gamma)),t) &= \sum_{S \subset \hp(\gamma)} (-1)^{\#S} \PP_{\gamma}(t;S=t) \\
    &= \sum_{T \subset \hp(\gamma)} \PP_{\gamma}(t; T=t, \hp(\gamma)\setminus T<t) \underbrace{\sum_{S \subset T} (-1)^{\# S} }_{(1-1)^{\#T}} = \PP_{\gamma}(t; \hp(\gamma)<t), 
\end{align*}
which leaves only the plane partitions with no high peaks having values $t$.

The ribbon-plane partitions counted in $\PP_{\gamma}(t;\hp<t)$ may also have entries less than $t$ on the elements of the order ideal $\langle \hp(\gamma)\rangle$ and may still have other entries equal to $t$ on the remaining shape $\gamma\setminus \langle \mathsf{hp}(\gamma)\rangle$. The entries equal to $t$ form an order filter (an upward closed subset). By specifying the order filter $F$ of  $\gamma\setminus \langle \mathsf{hp}(\gamma)\rangle$ with all the entries that are $t$, we have that 
\[
\PP_{\gamma}(t;\hp(\gamma)<t) = \sum_{ F \subset \gamma\setminus \langle \mathsf{hp}(\gamma) \rangle} \PP_{\gamma}(t;\gamma\setminus F<t,F=t),
\]
where the sum is over order ideals $F$ of $\gamma \setminus \langle \hp(\gamma)\rangle$ and $\PP_{\gamma}(t;\hp<t;F=t)$ is the number of plane partitions of shape $\gamma$ with entries $0,\ldots,t$ where the entries in $F$ have value $t$ and the rest of the entries in $\gamma\setminus F$ have values less than $t$. This quantity is equal to the number  of plane partitions of shape $\gamma\setminus  F$ with entries $0,\ldots,t-1$, which is counted by $\Omega(\gamma\setminus F;t)$. That is, $\PP_{\gamma}(t;\gamma\setminus F<t; F=t)\,=\, \Omega(\gamma \setminus F;t)$, see Figure~\ref{fig:decomp ribbons removing zeros} for an illustration of this reduction. 
\begin{figure}[ht]
\includegraphics{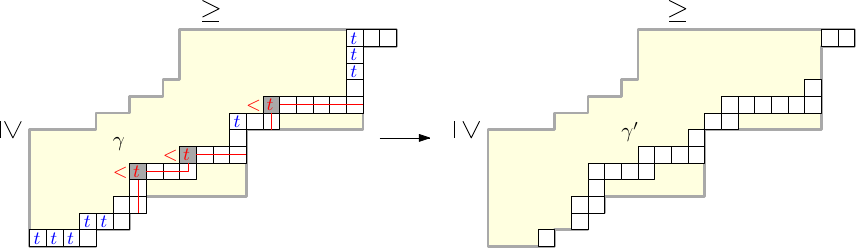}
\caption{Illustration of decomposition of plane partitions of ribbon shape $\gamma$ with entries $0,\ldots,t$ such that cells in high peaks, and the order ideal they generate have values $<t$ (\textcolor{red}{red}) and cells on an order filter $F$ have values $t$ (\textcolor{blue}{blue}) as a plane partition of (maybe disconnected) ribbon shape $\gamma\setminus F$ with entries $0,\ldots,t-1$.}
\label{fig:decomp ribbons removing zeros}
\end{figure}

By Theorem~\ref{coro:fences-order-positive-main}, this is a polynomial with nonnegative coefficients since $\gamma \setminus F$ is a border strip (fence).
\end{proof}

As a consequence of Theorem~\ref{thm: ehr +- fence is nonnegative } and its proof, we obtain the following formula for the polynomial $\ehr^{\pm}(\mathscr{P}(\M(\gamma)),t)$ as a positive sum of polynomials with nonnegative coefficients and prove Theorem~\ref{thm:main}.

\begin{theorem}\label{thm:formula-for-LPM}
    The Ehrhart polynomial of a lattice path matroid polytope $\mathscr{P}(\M(\lambda/\mu))$ can be written as follows:
\begin{equation} \label{eq:ehr LPM in terms of order filters}
\ehr(\mathscr{P}(\M(\lambda/\mu)),t) \,=\, \Omega(\gamma_{\min};t+1) + \sum_{\substack{\gamma \in \mathcal{L}(\lambda/\mu)\\ \gamma \neq \gamma_{min}}} \sum_{F\subset \gamma \setminus \langle \hp(\gamma)\rangle} \Omega(\gamma \setminus F; t),
\end{equation}
where the sum is over order filters $F$ of $\gamma \setminus \langle \hp(\gamma)\rangle$. In particular, since all the summands have nonnegative coefficients, the coefficients of $\ehr(\mathscr{P}(\M(\lambda/\mu)),t)$ are all positive. 
\end{theorem}

\begin{example}
 For the lattice path matroid $\M$ associated to $\lambda/\mu=433/1$, depending on the tails of the paths in $\mathcal{L}(\lambda/\mu)$, each of the paths $\gamma$ in Figure~\ref{fig: delannoy paths} will correspond to $1$, $6$, $3$, $2$, $1$ border strips $\gamma\setminus F$, respectively. See Figure~\ref{fig: positive subpaths} for a list of all the border strips associated to each path $\gamma$. 
 
 If we add the contributions of the border strips of each path $\gamma$, we obtain $\ehr^{\pm}(\mathscr{P}(\M(\gamma)),t)$. For instance, the path $\gamma_3$ corresponds to three border strips $\gamma_3 \setminus F$ and if we add their corresponding Ehrhart polynomials we obtain
\begin{align*}
\ehr^{\pm}(\mathscr{P}(\M(\gamma_3)),t) &= \Omega(431/2;t) + \Omega(331/2;t) + \Omega(31;t)\\
&=\PP_{431/2}(t-1) + \PP_{431/3}(t-1)+\PP_{31}(t-1)\\
&= \left(\frac{1}{18} t^{6} + \frac{7}{30} t^{5} + \frac{25}{72} t^{4} + \frac{1}{4} t^{3} + \frac{7}{72} t^{2} + \frac{1}{60} t\right)+\\&+\left(\frac{1}{8} t^{5} + \frac{5}{12} t^{4} + \frac{3}{8} t^{3} + \frac{1}{12} t^{2}\right) + \left(\frac{1}{8} t^{4} + \frac{5}{12} t^{3} + \frac{3}{8} t^{2} + \frac{1}{12} t\right)\\
&= \frac{1}{18} t^{6} + \frac{43}{120} t^{5} + \frac{8}{9} t^{4} + \frac{25}{24} t^{3} + \frac{5}{9} t^{2} + \frac{1}{10} t.
\end{align*}
 
\end{example}

\begin{figure}[ht]
    \includegraphics[scale=1.4]{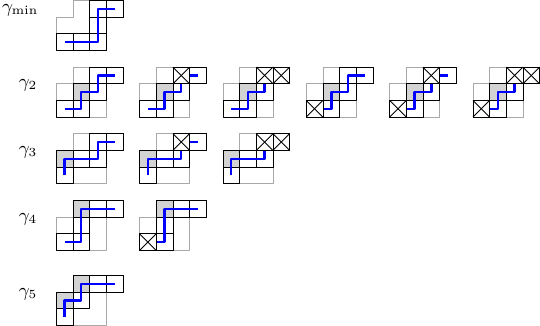}
    \caption{The sub-border strips that contribute in the positive formula for $\mathscr{P}(\M(\gamma))),t)$ for $\lambda/\mu=433/1$. The crossed cells correspond to the removed elements in the order filter $F$.}    \label{fig: positive subpaths}

\end{figure}

\begin{remark}
The formula in \eqref{eq:ehr LPM in terms of order filters} gives the Ehrhart polynomial of a lattice path matroid polytope as a sum of order polynomials of (not necessarily connected) fence posets. However, there is an asymmetry with the contribution from the $\gamma_{min}$ being evaluated at $t+1$ instead of $t$. As pointed by Kahane \cite[\S 6]{KahaneStat}, this assymetry can be resolved since by the ``coproduct formula" for order polynomials (see \cite[Ex.~1.2.4(d)]{kung-yan}) we have that 
\(
\Omega(\gamma_{\min};t+1) = \sum_{F \subset \gamma_{\min}} \Omega(\gamma\setminus F;t).
\)
Therefore, we can rewrite \eqref{eq:ehr LPM in terms of order filters} as
\begin{equation} \label{eq:ehr LPM in terms of order filters symmetric}
\ehr(\mathscr{P}(\M(\lambda/\mu)),t) \,=\, \sum_{\gamma \in \mathcal{L}(\lambda/\mu)} \sum_{F\subset \gamma \setminus \langle \hp(\gamma)\rangle} \Omega(\gamma \setminus F; t),
\end{equation}
\end{remark}

Recently, in \cite[Prop. 6.1]{KahaneStat} Kahane took \eqref{eq:ehr LPM in terms of order filters symmetric} and realized that the resulting fences in the expression are in correspondence with vertices of the original lattice path matroid polytope as follows. Given a lattice path $P$ lying above $L$ and below $U$, let $\sh(P)$ be the cells of $\lambda/\mu$ whose NorthWest corners touch $P$ called the {\em shade} of $P$. The cells $\sh(P)$ form a (not necessarily connected) border strip inside $\lambda$. See Example~\ref{ex:kahane reformulation}.

\begin{corollary}[{Kahane \cite{KahaneStat}}] \label{cor:Kahane reformulation Ehrhart LPM}
The Ehrhart polynomial of a lattice path matroid polytope $\mathscr{P}(\M[L,U])$ can be written as follows:
\[
\ehr(\mathscr{P}(\M[L,U],t)  \,=\, \sum_{P} \Omega(\sh(P);t),
\]
where the sum is over lattice paths $P$ lying above $L$ and below $U$.
\end{corollary}

\begin{example} \label{ex:kahane reformulation}
For the uniform matroid $\mathsf{U}_{2,4}$, which corresponds to the lattice path matroid of the shape $\lambda=(2,2)$, or equivalently, $\M[L,U]$ where $L=EENN$ and $U=NNEE$, the six paths $P$ above $L$ and below $U$ and their respective shades (highlighted in blue) are :
\begin{center}
\includegraphics[scale=0.7]{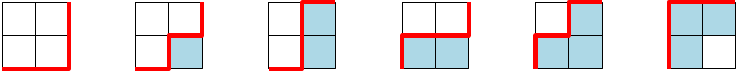} 
\end{center}
These shades form the border strips $22/22, 22/21, 22/11, 22/2, 22/1, 21$, respectively, thus
\[
\ehr(\mathscr{P}(\mathsf{U}_{2,4}),t) =  1 + \Omega(22/21;t) +  \Omega(22/11;t) +\Omega(22/2;t) + \Omega(22/1;t)+ \Omega(21;t).
\]
\end{example}

\section{Refinements and special cases}

\subsection{Combinatorial rule for the coefficients} 

Our proof of the Ehrhart positivity of the matroid base polytope of a lattice path matroid $\M(L,U)$ with ground set $[n]$, relied on expressing the Ehrhart polynomial as a sum of order polynomials of fence posets (see \eqref{eq:ehr LPM in terms of order filters}) which each are polynomials with nonnegative coefficients. The original proof of the positivity of $|\gamma|!\cdot \Omega(\gamma;t)$ in \cite{ferroni-morales-panova} was recursive and did not give a combinatorial interpretation for the coefficients of this normalized order polynomial though such an interpretation was recently found independently by Kahane \cite{Kahane} \cite{KahaneStat} and Huang \cite{huang-stat}. It would be interesting to find a combinatorial interpretation for the coefficients of $n!\cdot \ehr(\mathscr{P}(\M[L,U],t) \in \mathbb{N}[t]$. Note, that even for the hypersimplex $\Delta_{k,n}$, this is stated as an open problem in \cite[Exercise 4.62(h)]{EC2}. Recently, after the first version of this article was posted,  in \cite{KahaneStat} Kahane used his block statistic for the order polynomials of fences and his reformulation (see Corollary~\ref{cor:Kahane reformulation Ehrhart LPM}) of \eqref{eq:ehr LPM in terms of order filters} go give a combinatorial interpretation for the coefficients of $n!\cdot \ehr(\mathscr{P}M[L,U],t)$.

\subsection{Coefficient-wise inequality between lattice path matroids}

In \cite[Conjecture~1.5]{ferroni_hooks}, Ferroni conjectured  that a connected matroid $\M$ of rank $k$ on $n$ elements has an Ehrhart polynomial whose sequence of coefficients is entrywise bigger than the Ehrhart polynomial of the so-called minimal matroid, and entrywise smaller than the Ehrhart polynomial of the uniform matroid (in both cases, of the same size and rank as $\M$). The lower bound part was disproved by the same examples of matroids in \cite{ferroni} that are not Ehrhart positive, but the upper bound part is still conjectured to be true. The following implies that both the lower and upper bound parts of the conjecture hold for all lattice path matroids.

\begin{theorem}
Consider skew shapes $\eta/\nu \subseteq \lambda/\mu  \subseteq (n-k)^k$, then \[
 \ehr(\mathscr{P}(\M(\eta/\nu)),t)\,\preceq\, \ehr(\mathscr{P}(\M(\lambda/\mu)),t),\]
where the symbol $\preceq$ stands for coefficient-wise inequality.
\end{theorem}

\begin{proof}
It suffices to show the result for the case when $\lambda=\eta$ and $\nu\subset \mu$ Indeed, by reversing the lower and upper path of the lattice path matroid, the result follows for the case when $\mu=\nu$ and $\eta \subseteq \lambda$. The general result is then obtained by combining these two cases.

Note that when $\lambda=\eta$ and $\nu\subset \mu$ then $\mathcal{L}(\eta/\nu)\subset \mathcal{L}(\lambda/\mu)$ and moreover the high peaks of $\gamma$ in $\mathcal{L}(\eta/\nu)$ are the same as the excited peaks of $\gamma \in \mathcal{L}(\eta/\nu)$. Thus, by Theorem~\ref{thm: 2nd reformulation Ehrh LPMs} we have that 
\[
\ehr(\mathscr{P}(\M(\lambda/\mu)),t) - \ehr(\mathscr{P}(\M(\eta/\nu))  \,=\, \sum_{\gamma \in \mathcal{L}(\lambda/\mu)\setminus \mathcal{L}(\eta/\nu)} \ehr^{\pm}(\mathscr{P}(\M(\gamma)),t).
\]
The result then follows since by Theorem~\ref{thm: ehr +- fence is nonnegative }, each of the summands on the RHS is a polynomial with nonnegative coefficients.
\end{proof}

If one applies the preceding result to the special case of a direct sum of two uniform matroids $\mathsf{U}_{k-r,n-h}\oplus \mathsf{U}_{r,h}$ whose skew shape is contained in the corresponding skew shape of the \emph{cuspidal matroid} $\mathsf{\Lambda}_{r,k,h,n}$ (see \cite[Definition~3.21]{ferroni-schroter}), one obtains that 
    \[ \ehr(\mathscr{P}(\mathsf{U}_{k-r,n-h}),t) \cdot \ehr(\mathscr{P}(\mathsf{U}_{r,h}),t) \preceq \ehr(\mathscr{P}(\mathsf{\Lambda}_{r,k,h,n}),t).  \]

This allows us to use \cite[Theorem~1.5]{ferroni-schroter} to deduce the following result on elementary split matroids, which generalizes another main result of \cite{deligeorgaki-mcginnis-vindas} (note that elementary split matroids are a larger class than paving matroids).

\begin{corollary}
    If $\M$ is an elementary split matroid of rank $k$ on $n$ elements then:
    \[ \ehr(\mathscr{P}(\M), t) \preceq \ehr(\mathscr{P}(\mathsf{U}_{k,n}),t).\]
\end{corollary}

The above result supports the upper bound part of \cite[Conjecture~1.5]{ferroni_hooks}. Even though the lower bound part of that conjecture has been disproved by the examples in \cite{ferroni}, we still believe that the upper bound part holds for all matroids.

\subsection{The hypersimplex}

By specializing Theorem~\ref{thm:formula-for-LPM} for the uniform matroid $\mathsf{U}_{k,n}$ of rank $k$ and size $n$, which corresponds to the lattice path matroid of the shape $\lambda=(n-k)^k$ we have the following explicit formula.

\begin{proposition}\label{prop:uniform}
For $\lambda = (n-k)^k$ we have that
\begin{equation} \label{eq: equation Ehrhart uniform matroid}
\ehr(\mathscr{P}(\M(\lambda)),t) \,=\, \Omega(\gamma_{\min};t+1)+ \sum_{\gamma:(k,1) \to (1,n-k), \gamma \neq \gamma_{\min}} \sum_{i,j} \Omega(\gamma(i,j);t),
\end{equation}
where $\gamma(i,j)$ is the subpath of $\gamma$ obtained by removing $i$ initial east steps and $j$ final north steps in the case when $\gamma$ starts with at least $i$ east steps and ends with at least $j$ north steps.
\end{proposition}

This result reproves the main result of Ferroni \cite{ferroni_hypersimplex}, the Ehrhart positivity of the \emph{ hypersimplex} $\Delta_{k,n}$, in a completely different way. The base polytope of $\mathsf{U}_{k,n}$ is commonly called the hypersimplex $\Delta_{k,n}$, and the Ehrhart positivity problem for these polytopes traces back to a problem originally listed as open in Stanley's book \cite[Problem~4.62]{ec1}. We note that after \cite{ferroni_hypersimplex}, other different proofs of this positivity have been found, notably by Hanely et al. \cite{hanely} and by McGinnis \cite{mcginnis}.

\begin{remark}
It would be interesting to relate the existing formulas for the Ehrhart coefficients of $\Delta_{k,n}$ with the ones that result from Proposition~\ref{prop:uniform}. For example, if we take the leading coefficient on both sides we recover, up to multiplying by $n!$, the classical identity for Eulerian numbers $A_{n,k+1}$ as the sum of all the numbers of permutations with a given set of $k$ descents. Note that in recent work, Kahane \cite{KahaneStat} and independently Huang \cite{huang-stat} have {\em the same} combinatorial rule for the coefficients of $|\gamma|!\cdot \PP_{\gamma}(t)$ for ribbon shapes $\gamma$.
\end{remark}

\begin{example}
For the uniform matroid $\mathsf{U}_{2,4}$, which corresponds to the lattice path matroid of the shape $\lambda=(2,2)$, the above formula  \eqref{eq: equation Ehrhart uniform matroid}  will have contributions from the paths $\gamma_{\min}$,  and  $(2,1)-(1,1)-(1,2)$. Thus
\begin{align*}
 \ehr(\mathscr{P}(\mathsf{U}_{2,4}),t) &= \PP_{22/1}(t) + \PP_{21}(t-1) \\
\frac23t^3 + 2t^2 + \frac73t + 1  &=
 \frac{(t+1)(t+2)(2t+3)}{6} + \frac{t(t+1)(2t+1)}{6}.
\end{align*}

\end{example}

\begin{example}
For the uniform matroid $\mathsf{U}_{3,6}$, which corresponds to the lattice path matroid of the shape $\lambda=(3,3,3)$, the above formula \eqref{eq: equation Ehrhart uniform matroid} will have contributions from six paths including $\gamma_{\min}$ as illustrated in Figure~\ref{fig:shape 333}. 
\begin{figure}[ht]
    \includegraphics[scale=1.4]{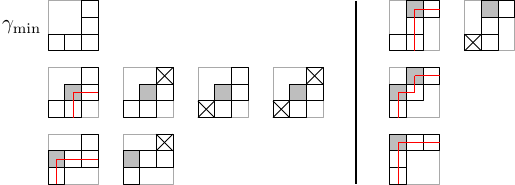}
    \caption{The border strips appearing in the formula for the Ehrhart polynomial of the uniform matroid $\mathsf{U}_{3,6}$. The crossed cells correspond to the removed elements in the order filter. The cells in \textcolor{red}{red} correspond to the elements of the order ideal generated by the high peaks of each lattice path.}
    \label{fig:shape 333}
\end{figure}
\begin{align*}
\ehr(\mathscr{P}(\mathsf{U}_{3,6}),t) &= \PP_{333/22}(t) + \left(\PP_{332/21}(t-1) + \PP_{32/1}(t-1) + \PP_{221/1}(t-1) + \PP_{21}(t-1)\right) \\
&+ \left(\PP_{331/2}(t-1) + \PP_{31}(t-1) \right) +  \\
&+\left(\PP_{322/11}(t-1) + \PP_{211}(t-1)\right) + \PP_{321/1}(t-1) + \PP_{311}(t-1)\\
&=
\frac{1}{20}t^5+\frac{1}{2}t^4+\frac{23}{12}t^3+\frac{7}{2}t^2+\frac{91}{30}t+1
 + \left(\frac{2}{15}t^5+\frac{3}{4}t^4+\frac{3}{2}t^3+\frac{5}{4}t^2+\frac{11}{30}t\right)\\
&+2\left(\frac{11}{120}t^5+\frac{11}{24}t^4+\frac{19}{24}t^3+\frac{13}{24}t^2+\frac{7}{60}t\right)+\\
&+\left(\frac{2}{15}t^5+\frac{1}{3}t^4+\frac{1}{3}t^3+\frac{1}{6}t^2+\frac{1}{30}t\right)
+\left(\frac{1}{20}t^5+\frac{1}{4}t^4+\frac{5}{12}t^3+\frac{1}{4}t^2+\frac{1}{30}t\right)\\
&= \frac{11}{20}t^5 + \frac{11}{4}t^4 + \frac{23}{4}t^3 + \frac{25}{4}t^2 + \frac{37}{10}t + 1.
\end{align*}
\end{example}

\subsection{A remark towards generalizations}

As explained in the introduction of this paper, our positivity result does not cover the case of series-parallel matroids. However, a key tool we employed in the main proof of this article is that (connected) lattice path matroids admit nice subdivisions into snake matroids, a special class of series-parallel matroids. All positroids can be subdivided into series-parallel matroids (see for example Speyer and Williams \cite{speyer-williams}). It is reasonable to inquire whether the existence of series-parallel subdivisions is enough to ensure Ehrhart positivity. Not only would this prove Conjecture~\ref{conj:positroids}, but it would also prove that transversal matroids (and therefore also cotransversal matroids) are Ehrhart positive; note that transversal matroids and can be subdivided into series-parallel matroids thanks to results by Fink and Rinc\'on \cite{fink-rincon}. Motivated by a similar conjecture formulated in the context of Speyer's $g$-polynomial in \cite[Conjecture~4.6]{luis_delannoy}, we are led to propose the following conjecture.

\begin{conjecture}
    Let $\M$ be a connected matroid whose base polytope can be subdivided into series-parallel matroids. Then, $\M$ is Ehrhart positive.
\end{conjecture}

As said above, this is strictly stronger than Conjecture~\ref{conj:positroids}. The class of matroids that can be subdivided into (direct sums of) series-parallel matroids is easily seen to be closed under minors and duality, so this would actually imply that all gammoids are Ehrhart positive (gammoids are the smallest class of matroids closed under minors and duality and containing transversal matroids). We note that there are non-gammoids that admit a series-parallel subdivision: in a private communication with the authors David Speyer described one such subdivision for the the V\'amos matroid, which is certainly not a gammoid as it is not realizable.

\bibliographystyle{amsalpha}
\bibliography{bibliography_LPMs}

\end{document}